\documentclass[11pt]{amsart}

\usepackage{amsmath}
\usepackage{amssymb}
\usepackage{amscd}
\usepackage{color}

\begin{document}

\title{Quasipolarity of Generalized Matrix Rings}

\author{Orhan Gurgun}
\address{Orhan Gurgun, Department of Mathematics, Ankara University,  Turkey}
\email{orhangurgun@gmail.com}

\author{Sait Halicioglu}
\address{Sait Hal\i c\i oglu,  Department of Mathematics, Ankara University, Turkey}
\email{halici@ankara.edu.tr}

\author{Abdullah Harmanci}
\address{Abdullah Harmanci, Department of Maths, Hacettepe University,  Turkey}
\email{harmanci@hacettepe.edu.tr}

\date{\empty}
\date{}
\newtheorem{thm}{Theorem}[section]
\newtheorem{lem}[thm]{Lemma}
\newtheorem{prop}[thm]{Proposition}
\newtheorem{cor}[thm]{Corollary}
\newtheorem{exs}[thm]{Examples}
\newtheorem{defn}[thm]{Definition}
\newtheorem{nota}{Notation}
\newtheorem{rem}[thm]{Remark}
\newtheorem{ex}[thm]{Example}

\maketitle

\begin{abstract} An element $a$ of a ring $R$ is called \emph{quasipolar} provided that there
exists an idempotent $p\in R$ such that $p\in comm^2(a)$, $a+p\in
U(R)$ and $ap\in R^{qnil}$. A ring $R$ is \emph{quasipolar} in
case every element in $R$ is quasipolar. In this paper, we
investigate quasipolarity of generalized matrix rings $K_s (R)$
for a commutative local ring $R$ and $s\in R$. We show that if $s$
is nilpotent, then $K_s(R)$ is quasipolar. We determine the
conditions under which elements of $K_s (R)$ are quasipolar. It is
shown that $K_s(R)$ is quasipolar if and only if $tr(A)\in J(R)$
or the equation $x^2-tr(A)x+det_s(A)=0$ is solvable in $R$ for
every $A\in K_s(R)$ with $det_s(A)\in J(R)$. Furthermore, we prove
that $M_2(R)$ is quasipolar if and only if $M_2(R)$ is strongly
clean for a commutative local ring $R$.
 \\[+2mm]
\end{abstract}

\section{Introduction}
Throughout this paper all rings are associative with identity
unless otherwise stated. Following Koliha and Patricio \cite{KP},
the \emph{commutant} and \emph{double commutant} of $a\in R$ are
defined by $comm(a)=\{x\in R~|~xa=ax\}$, $comm^2(a)=\{x\in
R~|~xy=yx~\mbox{for all}~y\in comm(a)\}$, respectively. If
$R^{qnil}=\{a\in R~|~1+ax\in U(R)~\mbox{for every}~x\in comm(a)\}$
and $a\in R^{qnil}$,  then $a$ is said to be \emph{quasinilpotent}
\cite{H}. An element $a\in R$ is called \emph{quasipolar} provided
that there exists an idempotent $p\in R$ such that $p\in
comm^2(a)$, $a+p\in U(R)$ and $ap\in R^{qnil}$. A ring $R$ is
\emph{quasipolar} in case every element in R is quasipolar. Some
properties of quasipolar rings were studied in \cite{CC2, CC1, KP,
YC}.

Let $R$ be a ring, and let $s\in R$ be central. Following
Krylov~\cite{K}, we use $K_s(R)$ to denote the set $\{ [a_{ij}]\in
M_2(R)~|~\mbox{each}~a_{ij}\in R\}$ with the following operations:

 $$\begin{array}{c}
\left[
\begin{array}{cc}
a&b\\
c&d
\end{array}
\right]+\left[
\begin{array}{cc}
a'&b'\\
c'&d'
\end{array}
\right]=\left[
\begin{array}{cc}
a+a'&b+b'\\
c+c'&d+d'
\end{array}
\right],\\
\left[
\begin{array}{cc}
a&b\\
c&d
\end{array}
\right]\left[
\begin{array}{cc}
a'&b'\\
c'&d'
\end{array}
\right]=\left[
\begin{array}{cc}
aa'+sbc'&ab'+bd'\\
ca'+dc'&scb'+dd'
\end{array}
\right].
\end{array}$$

\noindent The element $s$ is called \textit{the multiplier of}
$K_s(R)$. The set $K_s(R)$ becomes a ring with these operations
and can be viewed as a special kind of Morita context. A Morita
context $(A, B, M, N, \psi, \phi)$ consists of two rings $A$ and
$B$, two bimodules $_AM_B$, $_BN_A$ and a pair of bimodule
homomorphisms $\psi: M\bigotimes_B N\rightarrow A$ and $\phi:
N\bigotimes_A M\rightarrow B$ which satisfy the following
associativity: $\psi(m\bigotimes n)m'=m\phi(n\bigotimes m')$ and
$\phi(n\bigotimes m)n'=n\psi(m\bigotimes n')$ for any $n, n'\in
N$, $m, m'\in M$. These conditions insure that the set $T$ of
generalized matrices $\left[%
\begin{array}{cc}
  a & m \\
  n & b \\
\end{array}%
\right]$; $a\in A$, $b\in B$, $m\in M$, $n\in N$ will form a ring,
called the ring of the \emph{Morita context}. A Morita context $\left[%
\begin{array}{cc}
  A & M \\
  N & B \\
\end{array}%
\right]$ with $A = B = M = N = R$ is called \textit{a generalized
matrix ring over} $R$. It was observed by Krylov \cite{K} that the
generalized matrix rings over $R$ are precisely these rings
$K_s(R)$ with $s\in C(R)$. When $s = 1$, $K_1(R)$ is just the
matrix ring $M_2(R)$, but $K_s(R)$ can be significantly different
from $M_2(R)$. In fact, for a local ring $R$ and $s\in C(R)$,
$K_s(R)\cong K_1(R)$ if and only if $s\in U(R)$ (see \cite[Lemma 3
and Corollary 2]{K} and \cite[Corollary 4.10]{TLZ}). Some
properties of the ring $K_s(R)$ is studied comprehensively by
Krylov and Tuganbaev   in \cite{KT}.

 In this paper, we study the
quasipolarity of the generalized matrix ring $K_s (R)$  over a
commutative local ring $R$. It is shown that $K_s(R)$ is
quasipolar if and only if $tr(A)\in J(R)$ or the equation
$x^2-tr(A)x+det_s(A)=0$ is solvable in $R$ for every $A\in K_s(R)$
with $det_s(A)\in J(R)$. This yields the main result of \cite{CC2}
for $s=1$. Furthermore, for $s\in U(R)$, we show that $K_s(R)$ is
quasipolar if and only if $K_s(R)$ is strongly clean. In
particular, a criterion for the quasipolarity of
$K_s\big(R[[x]]\big)$ is obtained. We see that if $s$ is a
nilpotent element in $R$, then $K_s(R)$ is quasipolar.

In what follows,  the ring of integers modulo $n$ is denoted by
$\mathbb{Z}_n$, and we write $M_n(R)$ (resp. $T_n(R)$) for the
rings of all (resp., all upper triangular) $n\times n$ matrices
over the ring $R$.  We write $R[[x]]$, $U(R)$, $C(R)$, $R^{nil}$
and $J(R)$ for the power series ring over a ring $R$, the set of
all invertible elements, the set of all central elements, the set
of all nilpotent elements and the Jacobson radical of $R$,
respectively. For a prime integer $p$, $\mathbb{Z}_{(p)}$ will be
the localization of $\mathbb{Z}$ at $p$ and the ring
$\widehat{\Bbb Z}_p$ denotes $p$-adic integers.

\section{Generalized Matrix Rings} In this section we study the quasipolarity of the generalized matrix ring
over a commutative local ring. The quasipolarity of $K_s(R)$
strictly depends on the values of $s$. Namely,  $s \in U(R)$ if
and only if $K_s(R)\cong M_2(R)$ for any local ring $R$ with $s
\in C(R)$. We prove that if  $s$ is a nilpotent element of $R$,
then $K_s(R)$ is always quasipolar. We supply  an example to show
that $K_s(R)$ is quasipolar if and only if $s\in R^{nil}$ (see
Example~\ref{ex1}). And if $s \in U(R)$, then $K_s(\Bbb Z_{(p)})$
is not quasipolar.

\begin{lem} \label{lem7} \cite[Lemma 1]{TZ} Let $R$ be a ring and let $s\in C(R)$. Then $\left[%
\begin{array}{cc}
 a & x\\
  y & b \\
\end{array}%
\right] \mapsto \left[%
\begin{array}{cc}
  b & y \\
  x & a \\
\end{array}%
\right]$ is an automorphism of $K_s(R)$.
\end{lem}

If $R$ is a commutative ring with $s\in R$ and $A=\left[%
\begin{array}{cc}
 a & b\\
 c & d \\
\end{array}%
\right]\in K_s(R)$, we define $det_s(A)=ad-sbc$ and $tr(A)=a+d$,
and $rA=\left[%
\begin{array}{cc}
 ra & rb\\
 rc & rd\\
\end{array}%
\right]$ for $r\in R$ (see \cite{TZ}).

For elements $a,b$ in a ring $R$, we use the notation $a\sim b$ to
mean that $a$ \textit{is similar to} $b$, that is, $b=u^{-1}au$
for some $u\in U(R)$.

\begin{lem}\label{temel1}\cite[Lemma 14]{TZ} Let $R$ be a commutative ring with $s\in R$ and let $A, B\in K_s(R)$. The following hold:
\begin{itemize}
\item[(1)] $det_s(AB)=det_s(A)det_s(B)$.
\item[(2)] $A\in U\big(K_s(R)\big)$ if and only if $det_s(A)\in
U(R)$. In this case, if $A=\left[%
\begin{array}{cc}
 a & b\\
 c & d \\
\end{array}%
\right]$, then $A^{-1}= det_s(A)^{-1}\left[%
\begin{array}{rr}
 d & -b\\
 -c & a \\
\end{array}%
\right]$.
\item[(3)] If $A\sim B$, then $det_s(A)=det_s(B)$ and $tr(A)=tr(B)$.
\end{itemize}
\end{lem}

Note that $det_s(I_2+A)=1+tr(A)+det_s(A)$ and
$A^2-tr(A)A+det_s(A)I_2=0$ for any $A\in K_s(R)$.

We need Theorem \ref{teo1} in the sequel and give a short proof
for the sake of completeness.

\begin{thm}\label{teo1} Let $R$ be a commutative local ring with $s\in R$ and let $A\in K_s(R)$ such that
$det_s(A)\in J(R)$. Then $tr(A)\in J(R)$ if and only if $A\in
\big(K_s(R)\big)^{qnil}$.
\end{thm}

\begin{proof} Let $A\in K_s(R)$ such that $det_s(A)\in J(R)$. Assume that $tr(A)\in
J(R)$. Since $A^2-tr(A)A+det_s(A)I_2=0$, we have
$A^2=tr(A)A-det_s(A)I_2\in J\big(K_s(R)\big)$. Let $X\in comm(A)$.
Then $(I_2-XA)(I_2+XA)=I_2-X^2A^2\in U(R)$ and so $I_2-XA\in
U\big(K_s(R)\big)$ because $A^2\in J\big(K_s(R)\big)$. Hence $A\in
\big(K_s(R)\big)^{qnil}$. Conversely, suppose that $A\in
\big(K_s(R)\big)^{qnil}$ and let $x\in R$. Since $A\in
\big(K_s(R)\big)^{qnil}$, $I_2+xA\in U\big(K_s(R)\big)$ and so
$det_s(I_2+xA)=1+xtr(A)+x^2det_s(A)\in U(R)$. As $det_s(A)\in
J(R)$, we have $1+xtr(A)\in U(R)$. This gives $tr(A)\in J(R)$.
\end{proof}

\begin{lem} \label{lem5} Let $R$ be a ring and $s\in C(R)$. Then $A\in K_s(R)$ is
quasipolar if and only if $PAP^{-1}\in K_s(R)$ is quasipolar for
some $P\in U\big(K_s(R)\big)$.
\end{lem}

\begin{proof} It follows from  \cite[Lemma 2.3]{CC2}.
\end{proof}

For elements $a,b$ in a ring $R$, we say that \textit{$a$ is
equivalent to
$b$} if there exist $u,v\in U(R)$ such that $b=uav$. Recall that an element $\left[%
\begin{array}{cc}
 a & 0\\
  0 & b \\
\end{array}%
\right]$ is called \textit{a diagonal matrix of} $K_s(R)$.

\begin{lem} \label{lem3} \cite[Lemma 3]{TZ} Let $E^2=E\in K_s(R)$. If $E$ is equivalent to a diagonal matrix
in $K_s(R)$, then  $E$ is similar to a diagonal matrix in
$K_s(R)$.
\end{lem}

\begin{lem} \label{lem2} Let $R$ be a local ring with $s\in C(R)$ and let $E$ be a non-trivial idempotent of $K_s(R)$.
Then we have the following.

\begin{itemize}
\item[(1)] If $s\in U(R)$, then $E\sim \left[%
\begin{array}{cc}
 1 & 0\\
  0 & 0 \\
\end{array}%
\right]$.
\item[(2)] If $s\in J(R)$, then either $E\sim \left[%
\begin{array}{cc}
 1 & 0\\
  0 & 0 \\
\end{array}%
\right]$ or $E\sim \left[%
\begin{array}{cc}
 0 & 0\\
  0 & 1 \\
\end{array}%
\right]$.
\end{itemize}
\end{lem}

\begin{proof} Let $E=\left[%
\begin{array}{cc}
 a & b\\
  c & d \\
\end{array}%
\right]$ where $a,b,c,d\in R$. Since $E^2=E$, we have
$$a^2+sbc=a, ~scb+d^2=d, ~ab+bd=b, ~ca+dc=c.$$

\noindent If $a,d\in J(R)$, then $scb, sbc, b,c\in J(R)$ and so
$E\in J\big(K_s(R)\big)$. Hence $E=0$, a
contradiction. Since $R$ is local, we have $a\in U(R)$ or $d\in U(R)$. If $d\in U(R)$, then $\left[%
\begin{array}{cc}
 1 & -bd^{-1}\\
  0 & 1 \\
\end{array}%
\right]\left[%
\begin{array}{cc}
 a & b\\
  c & d \\
\end{array}%
\right]\left[%
\begin{array}{cc}
 1 & 0\\
  -d^{-1}c & d^{-1} \\
\end{array}%
\right]=\left[%
\begin{array}{cc}
 a-sbd^{-1}c & 0\\
  0 & 1 \\
\end{array}%
\right]$ and so $E$ is equivalent to a diagonal matrix. If $a\in
U(R)$, then we similarly show that $E$ is equivalent to a diagonal
matrix (see \cite[Lemma 4]{TZ}). According to Lemma~\ref{lem3},
there exists $P\in U\big(K_s(R)\big)$ such that $PEP^{-1}=\left[%
\begin{array}{cc}
 f & 0\\
  0 & g \\
\end{array}%
\right]$, where $f,g\in R$ are idempotents. Since $R$ is local and
$E\neq 0$, we see that either $f=1$ and $g=0$ or $f=0$ and $g=1$.
If $f=1$ and $g=0$, then $E\sim \left[%
\begin{array}{cc}
 1 & 0\\
  0 & 0 \\
\end{array}%
\right]$. Let $f=0$ and $g=1$. If $s\in U(R)$, then $\left[%
\begin{array}{cc}
 0 & 1\\
  1 & 0 \\
\end{array}%
\right]\left[%
\begin{array}{cc}
 0 & 0\\
  0 & 1 \\
\end{array}%
\right]\left[%
\begin{array}{cc}
 0 & s^{-1}\\
  s^{-1} & 0 \\
\end{array}%
\right]=\left[%
\begin{array}{cc}
 1 & 0\\
  0 & 0 \\
\end{array}%
\right]$ and so $E\sim \left[%
\begin{array}{cc}
 1 & 0\\
  0 & 0 \\
\end{array}%
\right]$. If $s\in J(R)$ and $E\sim \left[%
\begin{array}{cc}
 1 & 0\\
  0 & 0 \\
\end{array}%
\right]$, then it is easy to check that $s\in U(R)$, a
contradiction. Hence, if $s\in J(R)$, then either $E\sim \left[%
\begin{array}{cc}
 1 & 0\\
  0 & 0 \\
\end{array}%
\right]$ or $E\sim \left[%
\begin{array}{cc}
 0 & 0\\
  0 & 1 \\
\end{array}%
\right]$.
\end{proof}

\begin{lem}\label{lem4}Let $R$ be a uniquely bleached local ring and $s\in C(R)$. Then $A\in K_s(R)$ is quasipolar
if and only if either $A\in U\big(K_s(R)\big)$ or $A\in
\big(K_s(R)\big)^{qnil}$ or $A\sim \left[%
\begin{array}{cc}
 a & 0\\
 0 & b \\
\end{array}%
\right]$ where $a,b\in R$.
\end{lem}

\begin{proof} Suppose that $A$ is quasipolar in $K_s(R)$. Then there
exists $E^2=E\in comm^2(A)$ such that $A+E=W\in U\big(K_s(R)\big)$
and $EA\in \big(K_s(R)\big)^{qnil}$. If $E=0$ or $E=I_2$, then
$A\in U\big(K_s(R)\big)$ or $A\in \big(K_s(R)\big)^{qnil}$,
respectively. Hence, by Lemma~\ref{lem2}, either $E\sim \left[%
\begin{array}{cc}
 1 & 0\\
  0 & 0 \\
\end{array}%
\right]$ or $E\sim \left[%
\begin{array}{cc}
 0 & 0\\
  0 & 1 \\
\end{array}%
\right]$. Assume that $E\sim \left[%
\begin{array}{cc}
 0 & 0\\
  0 & 1 \\
\end{array}%
\right]$. That is, there exists $P\in
U\big(K_s(R)\big)$ such that $PEP^{-1}=\left[%
\begin{array}{cc}
 1 & 0\\
  0 & 0 \\
\end{array}%
\right]$. According to Lemma~\ref{lem5},
$PAP^{-1}+PEP^{-1}=PWP^{-1}$ is a quasipolar decomposition in
$K_s(R)$.
Let $V=[v_{ij}]=PWP^{-1}$. Since $V\left[%
\begin{array}{cc}
 1 & 0\\
  0 & 0 \\
\end{array}%
\right]=\left[%
\begin{array}{cc}
 1 & 0\\
  0 & 0 \\
\end{array}%
\right]V$, we have $v_{12}=v_{21}=0$ and so $A\sim \left[%
\begin{array}{cc}
 a & 0\\
 0 & b \\
\end{array}%
\right]$ where $a,b\in R$. Conversely, it suffices to show that
$B=\left[%
\begin{array}{cc}
 a & 0\\
 0 & b \\
\end{array}%
\right]$ is quasipolar in $K_s(R)$ by Lemma \ref{lem5}. Since $R$
is a local ring, $R$ is quasipolar by \cite[Corollary 3.3]{YC}.
Then there exist $e_1^2=e_1\in comm^2(a)$ and $e_2^2=e_2\in
comm^2(b)$ such that $a+e_1\in U(R)$, $b+e_2\in U(R)$, $ae_1\in
R^{qnil}$ and
$be_2\in R^{qnil}$. Let $F=\left[%
\begin{array}{cc}
 e_1 & 0\\
 0 & e_2 \\
\end{array}%
\right]$. It can be easily check that $F\in comm^2(B)$ and if
$X\in comm(BF)$, then $I_2+XBF\in U\big(K_s(R)\big)$
and so $BF\in \big(K_s(R)\big)^{qnil}$. This implies that $B+F=\left[%
\begin{array}{cc}
 a+e_1 & 0\\
 0 & b+e_2 \\
\end{array}%
\right]$ is a quasipolar decomposition.
\end{proof}

\begin{lem} \label{lem8} Let $R$ be a local ring with $s\in J(R)\cap C(R)$. Let $A\in K_s(R)$ such that
$A\notin U\big(K_s(R)\big)$ and $A\notin \big(K_s(R)\big)^{qnil}$.
Then $A$ is similar to $\left[%
\begin{array}{cc}
 u & 1\\
 v & w \\
\end{array}%
\right]$ or $\left[%
\begin{array}{cc}
 w & 1\\
 v & u \\
\end{array}%
\right]$ where $u,v\in U(R)$ and $w\in J(R)$.
\end{lem}

\begin{proof} It is similar to the proof of  \cite[Lemma 5]{TZ}.
\end{proof}

\begin{lem} \label{lem6} Let $R$ be a commutative local ring with $s\in R$. Then every upper triangular matrix in
$K_s(R)$ is quasipolar.
\end{lem}

\begin{proof} Let $A=\left[%
\begin{array}{cc}
 a & b\\
 0 & c \\
\end{array}%
\right]\in K_s(R)$ where $a,b,c\in R$. We can assume that
$det_s(A)=ac\in J(R)$ and $tr(A)=a+c\in U(R)$ by
Theorem~\ref{teo1}. This gives $c-a\in U(R)$. Choose $P=\left[%
\begin{array}{cc}
 1 & b(c-a)^{-1}\\
 0 & 1 \\
\end{array}%
\right]$. By Lemma~\ref{temel1}, $P\in U\big(K_s(R)\big)$,
and a direct calculation shows that $P^{-1}AP=\left[%
\begin{array}{cc}
 a & 0\\
 0 & c \\
\end{array}%
\right]$. Then $A$ is quasipolar from Lemma~\ref{lem5} and
Lemma~\ref{lem4}. \end{proof}

\begin{thm}\label{teo2}Let $R$ be a commutative local ring with $s\in R$ and let $A\in K_s(R)$ such that
neither $A\in U\big(K_s(R)\big)$ nor $A\in
\big(K_s(R)\big)^{qnil}$. Then the following statements are
equivalent:
\begin{enumerate}
\item $A$ is quasipolar in $K_s(R)$.
\item The equation $x^2-tr(A)x+det_s(A)=0$ is solvable in $R$.
\end{enumerate}
\end{thm}

\begin{proof}$(1)\Rightarrow (2)$ Assume that $A\in K_s(R)$ is quasipolar. Since $A\notin U\big(K_s(R)\big)$ and
$A\notin \big(K_s(R)\big)^{qnil}$, $A\sim B=\left[%
\begin{array}{cc}
 a & 0\\
 0 & b \\
\end{array}%
\right]$ where $a,b\in R$ by Lemma~\ref{lem4}. According to
Lemma~\ref{temel1}, $tr(A)=tr(B)$ and $det_s(A)=det_s(B)$. This
gives $x^2-tr(A)x+det_s(A)=x^2-tr(B)x+det_s(B)$. Since
$a^2-tr(B)a+det_s(B)=0$, the equation $x^2-tr(A)x+det_s(A)=0$ is
solvable in $R$.

$(2)\Rightarrow (1)$ Let $A=\left[%
\begin{array}{cc}
 a_{11} & a_{12}\\
 a_{21} & a_{22} \\
\end{array}%
\right]\in K_s(R)$ and suppose that the equation
$x^2-tr(A)x+det_s(A)=0$ has roots $a,b\in R$. Since $A\notin
U\big(K_s(R)\big)$ and $A\notin \big(K_s(R)\big)^{qnil}$,
$det_s(A)=a_{11}a_{22}-sa_{12}a_{21}=ab\in J(R)$ and
$tr(A)=a_{11}+a_{22}=a+b\in U(R)$ by Theorem~\ref{teo1}. So one of
$a, b$ must be in $U(R)$ and the other must be in $J(R)$.
 Let $B=\left[%
\begin{array}{cc}
 a_{22} & a_{21}\\
 a_{12} & a_{11} \\
\end{array}%
\right]$. Then $tr(B)=tr(A)$ and $det_s(B)=det_s(A)$, and $A$ is
quasipolar if and only if $B$ is quasipolar by Lemma~\ref{lem7}.
Hence, without loss of generality we may assume that $a\in U(R)$,
$b\in J(R)$ and $a_{11}\in U(R)$. Let $P=\left[%
\begin{array}{cr}
 1 & 0\\
 a_{21}(a-a_{22})^{-1} & 1 \\
\end{array}%
\right]$. By Lemma~\ref{temel1}(2),  $P\in U\big(K_s(R)\big)$ and
easy calculation shows that $P^{-1}AP$ is an upper triangular
matrix in $K_s(R)$. Therefore $A$ is quasipolar from
Lemma~\ref{lem5} and Lemma~\ref{lem6}.
\end{proof}

\begin{thm}\label{teo3} Let $R$ be a commutative local ring with $s\in R$.
The following are equivalent.
\begin{enumerate}
\item $K_s(R)$ is quasipolar.
\item For every $A\in K_s(R)$ with $det_s(A)\in J(R)$, one of the following holds:

\begin{itemize}
    \item[(i)] $tr(A)\in J(R)$,
    \item[(ii)] The equation $x^2-tr(A)x+det_s(A)=0$ is solvable in $R$.
\end{itemize}
\end{enumerate}
\end{thm}

\begin{proof}$(1)\Rightarrow (2)$ Suppose that $A\in K_s(R)$ with $det_s(A)\in
J(R)$. By $(1)$, there exists an idempotent $E\in K_s(R)$ such
that $E\in comm^2(A)$ and $A+E\in U\big(K_s(R)\big)$ and $EA\in
\big(K_s(R)\big)^{qnil}$. If $E=I_2$, then $A\in
\big(K_s(R)\big)^{qnil}$ and so $tr(A)\in J(R)$ by
Theorem~\ref{teo1}. So we can assume that $A\notin
\big(K_s(R)\big)^{qnil}$. Since $det_s(A)\in J(R)$, $A\notin
U\big(K_s(R)\big)$ by Lemma~\ref{temel1}(2). According to
Theorem~\ref{teo2}, the equation $x^2-tr(A)x+det_s(A)=0$ is
solvable in $R$.

$(2)\Rightarrow (1)$ Let $A\in K_s(R)$. If $det_s(A)\in U(R)$,
then $A\in U\big(K_s(R)\big)$ and so $A$ is quasipolar. Let
$det_s(A)\in J(R)$. If $tr(A)\in J(R)$, then $A$ is quasipolar by
Theorem~\ref{teo1}. Hence we assume that $tr(A)\in U(R)$. This
gives $A\notin U\big(K_s(R)\big)$ and $A\notin
\big(K_s(R)\big)^{qnil}$. By Theorem~\ref{teo2}, $A$ is quasipolar
and so $K_s(R)$ is quasipolar.
\end{proof}

Letting $s=1$ in Theorem~\ref{teo3} yields the main result of
\cite{CC2}.

\begin{cor}\label{cor4}Let $R$ be a commutative local ring
with $s\in U(R)$. The following are equivalent.
\begin{enumerate}
    \item $K_s(R)$ is quasipolar.
    \item $K_s(R)$ is strongly clean.
    \end{enumerate}
\end{cor}

\begin{proof} $(1)\Rightarrow(2)$ is obvious.
 $(2)\Rightarrow(1)$ Let $A\in K_s(R)$. According to
Theorem~\ref{teo3}, we may assume that $det_s(A)\in J(R)$ and
$tr(A)\in U(R)$. By \cite[Corollary 16]{TZ}, the equation
$t^2-t-w=0$ is solvable in $R$ for all $w\in J(R)$ . Then the
equation $t^2-t-\big(tr(A)\big)^{-2}det_s(A)=0$ is solvable in
$R$. Hence the equation $t^2-tr(A)t+det_s(A)=0$ is solvable in
$R$. In view of Theorem~\ref{teo3}, $K_s(R)$ is quasipolar.
\end{proof}

By using Corollary \ref{cor4}, from \cite[Theorem 2.4]{CYZ}, for
the ring $R=\widehat{\Bbb Z}_p$ and $s\in U(R)$, we have  $K_s(R)$
is quasipolar if and only if $K_s(R)$ is strongly clean.

\begin{cor}\label{tugce} Let $R$ be a commutative local ring. Then the
following are equivalent.
\begin{enumerate}
    \item $K_1(R)$ is quasipolar.
    \item $K_1(R)$ is strongly clean.
    \item $K_s(R)$ is strongly clean for all $s\in R$.
    \item $K_s(R)$ is strongly clean for all $s\in J(R)$.
\end{enumerate}
\end{cor}

\begin{proof} Combine Corollary~\ref{cor4} with  \cite[Corollary 19]{TZ}. \end{proof}

Corollary~\ref{tugce} shows that $M_2(R)$ is quasipolar if and
only if $M_2(R)$ is strongly clean for a commutative local ring
$R$.

\begin{thm}\label{teo6}Let $R$ be a commutative local ring with $s\in J(R)$. Then the following are equivalent.
\begin{enumerate}
    \item $K_s(R)$ is a quasipolar ring.
    \item For any $u,v\in U(R)$ and $w\in J(R)$, the equation
    $t^2-(u+w)t+(uw-sv)=0$ is solvable in $R$.
    \item For any $u\in U(R)$ and $w\in J(R)$, the equation
    $t^2-(1+w)t+(w-su)=0$ is solvable in $R$.
    \item $\left[%
\begin{array}{cc}
 1 & 1\\
 u & w \\
\end{array}%
\right]$ is quasipolar in $K_s(R)$ for all $u\in U(R)$ and $w\in
J(R)$.
\end{enumerate}
\end{thm}

\begin{proof} $(1)\Rightarrow (4)$ is obvious.

$(2)\Leftrightarrow(3)$ The equation
    $t^2-(u+w)t+(uw-sv)=0$ is solvable in $R$ if and only if the equation
    $(tu^{-1})^2-(1+wu^{-1})tu^{-1}+(wu^{-1}-svu^{-2})=0$ is solvable in $R$.

$(3)\Leftrightarrow(4)$  Follows from Theorem~\ref{teo2}.

$(2)\Rightarrow (1)$ Let $A\in K_s(R)$. According to
Theorem~\ref{teo3}, we may assume that $det_s(A)\in J(R)$ and
$tr(A)\in U(R)$, and then by Lemma~\ref{lem8}, similarly we can
assume that
$A\sim \left[%
\begin{array}{cc}
 u & 1\\
 v & w \\
\end{array}%
\right]$ where $u,v\in U(R)$ and $w\in J(R)$. Hence, by (2) and
Theorem~\ref{teo2}, $A$ is quasipolar and so holds (1).
\end{proof}

\begin{cor}\label{cor7} If $R=\widehat{\Bbb Z}_p$, then $K_s(R)$ is quasipolar for all $s\in R$.
\end{cor}

\begin{proof} We know that if $s\in
U(R)$, then $K_s(R)$ is quasipolar. Let $A\in K_s(R)$ and $s\in
J(R)$. In view of Theorem~\ref{teo6}, we can assume
that $A=\left[%
\begin{array}{cc}
 1 & 1\\
 u & w \\
\end{array}%
\right]$ where $u\in U(R)$ and $w\in J(R)$. Since $K_1(R)$ is
strongly clean, the equation $t^2-t-w=0$ is solvable in $R$ for
all $w\in J(R)$ by \cite[Corollary 16]{TZ}. Then the equation
$t^2-t+det_s(A)\big(tr(A)\big)^{-2}=0$ is solvable in $R$. This
gives that the equation $t^2-tr(A)t+det_s(A)=0$ is solvable in
$R$. Hence, by Theorem~\ref{teo2}, $A$ is quasipolar in $K_s(R)$.
Therefore $K_s(R)$ is quasipolar by Theorem~\ref{teo6}.
\end{proof}

\begin{lem}\label{root}  Let $R$ be a commutative local ring with
$s=\sum\limits_{i=0}^{\infty}s_ix^i\in R[[x]]$ and let $A(x)\in
K_s\big(R[[x]]\big)$. The following are equivalent.
\begin{itemize}
\item[(1)] The equation $t^2-tr\big(A(0)\big)t+det_{s_0}\big(A(0)\big)=0$ has a root in $J(R)$ and a root in $U(R)$.
\item[(2)] The equation $t^2-tr\big(A(x)\big)t+det_{s}\big(A(x)\big)=0$ has a root in $J\big(R[[x]]\big)$
and a root in $U\big(R[[x]]\big)$.
\end{itemize}
\end{lem}

\begin{proof} $(1)\Rightarrow (2)$ Note that $J\big(R[[x]]\big)=J(R)+xR[[x]]$. Assume that the equation
$t^2-tr\big(A(0)\big)t+det_{s_0}\big(A(0)\big)=0$ has a root
$\alpha\in J(R)$ and a root $\beta\in U(R)$. Let
$y=\sum\limits_{i=0}^{\infty}b_ix^i$,
$tr\big(A(x)\big)=\sum\limits_{i=0}^{\infty}\mu_ix^i$ and
$det_s\big(A(x)\big)=\sum\limits_{i=0}^{\infty}\lambda_ix^i\in
R[[x]]$ where $\mu_0=tr\big(A(0)\big)$ and
$\lambda_0=det_{s_0}\big(A(0)\big)$. Then,
$y^2-tr\big(A(x)\big)y-det_s\big(A(x)\big)=0$ holds in $R[[x]]$ if
the following equations are satisfied:

\[
\begin{aligned}
b_0^2-b_0\mu_0+\lambda_0=& 0\\
b_1[b_0+b_0-\mu_0]-b_0\mu_1+\lambda_1=&0\\
b_2[b_0+b_0-\mu_0]+b_1^2-b_0\mu_2-b_1\mu_1+\lambda_2=&0\\
\vdots
 \end{aligned}
\text{\quad ~~~~~\quad}
\begin{gathered}
(i_0)\\
(i_1)\\
(i_2)\\
\vdots
\end{gathered}\]

\noindent Obviously, $\mu_0=tr\big(A(0)\big)=\alpha+\beta\in
U(R)$. Let $b_0=\alpha$. Since $b_0+b_0-\mu_0=\alpha-\beta\in
U(R)$,  we obtain $b_1$, $b_2$, $b_3$, $\ldots$ from equations
above.  Then $t^2-tr\big(A(x)\big)t+det_{s}\big(A(x)\big)=0$ has a
root $\alpha(x)\in J\big(R[[x]]\big)$. If $b_0=\beta$,
analogously, we show that
$t^2-tr\big(A(x)\big)t+det_{s}\big(A(x)\big)=0$ has a root
$\beta(x)\in U\big(R[[x]]\big)$.

$(2)\Rightarrow (1)$ Suppose that the equation
$t^2-tr\big(A(x)\big)t+det_{s}\big(A(x)\big)$ has a root
$\alpha(x)\in J\big(R[[x]]\big)$ and a root $\beta(x)\in
U\big(R[[x]]\big)$. This implies that the equation
$t^2-tr\big(A(0)\big)t+det_{s_0}\big(A(0)\big)=0$ has a root
$\alpha(0)\in J(R)$ and a root $\beta(0)\in U(R)$.
\end{proof}

\begin{thm}\label{power}Let $R$ be a commutative local ring with
$s=\sum\limits_{i=0}^{\infty}s_ix^i\in R[[x]]$. The following are
equivalent.
\begin{enumerate}
\item $K_s\big(R[[x]]\big)$ is quasipolar.
\item $K_{s_{0}}(R)$ is quasipolar.
\end{enumerate}
\end{thm}

\begin{proof} $(1)\Rightarrow(2)$ Let $A\in K_{s_0}(R)$ with $det_{s_0}(A)\in J(R)$.
This gives $A\in K_s\big(R[[x]]\big)$ and $det_s(A)\in
J\big(R[[x]]\big)$. By $(1)$ and Theorem~\ref{teo3}, either
$tr(A)\in J\big(R[[x]]\big)$ or the equation
$t^2-tr(A)t+det_s(A)=0$ is solvable in $R[[x]]$. If $tr(A)\in
J\big(R[[x]]\big)$, then we have $tr(A)\in J(R)$ and so $A\in
K_{s_0}(R)$ is quasipolar by Theorem~\ref{teo3}. If the equation
$t^2-tr(A)t+det_s(A)=0$ is solvable in $R[[x]]$, then the equation
$t^2-tr(A)t+det_{s_0}(A)=0$ is solvable in $R$ by
Lemma~\ref{root}. So $A\in K_{s_0}(R)$ is quasipolar by
Theorem~\ref{teo3}.

$(2)\Rightarrow(1)$  Let $A(x)\in K_s\big(R[[x]]\big)$ with
$det_s\big(A(x)\big)\in J\big(R[[x]]\big)$. This gives
$det_{s_0}\big(A(0)\big)\in J(R)$. If $tr\big(A(x)\big)\in
J\big(R[[x]]\big)$, then $A(x)$ is quasipolar by
Theorem~\ref{teo2}. Hence we can assume that $tr\big(A(x)\big)\in
U\big(R[[x]]\big)$ and so $tr\big(A(0)\big)\in U(R)$. By $(2)$ and
Lemma~\ref{root}, the equation
$t^2-tr\big(A(0)\big)t+det_{s_0}\big(A(0)\big)=0$ is solvable in
$R$. Let $\lambda,\mu\in R$ be roots of this equation. Since
$tr\big(A(0)\big)\in U(R)$ and $det_{s_0}\big(A(0)\big)\in J(R)$,
one of $\lambda, \mu$ must be in $U(R)$ and the other must be in
$J(R)$. Without loss of generality we assume that $\lambda\in
J(R)$ and $\mu\in U(R)$. Thus the equation
$t^2-tr\big(A(x)\big)t+det_{s}\big(A(x)\big)=0$ is solvable in
$R[[x]]$ by Lemma~\ref{root}. According to Theorem~\ref{teo3},
$A(x)\in K_s\big(R[[x]]\big)$ is quasipolar and so holds $(1)$.
\end{proof}

\begin{cor}\label{cor2}Let $R$ be a commutative local ring with $s\in R$.
Then $K_s\big(R[[x]]\big)$ is quasipolar if and only if $K_s(R)$
is quasipolar.
\end{cor}

\begin{cor}\label{cor3}Let $R$ be a commutative local ring.
Then the following holds.
\begin{enumerate}
    \item $K_0(R)$ is quasipolar.
    \item $K_s(R[[x]])$ is quasipolar for all $s\in xR[[x]]$.
\end{enumerate}
\end{cor}

\begin{proof} $(1)$ Let $A=\left[%
\begin{array}{cc}
 a & b\\
 c & d \\
\end{array}%
\right]\in K_0(R)$ where $a,b,c,d\in R$. By Theorem~\ref{teo3}, we
can assume that $det_0(A)=ad\in J(R)$ and $tr(A)=a+d\in U(R)$.
This implies that $a-d\in U(R)$. Choose $P=\left[%
\begin{array}{cc}
 1 & 0\\
 c(a-d)^{-1} & -1 \\
\end{array}%
\right]$. Hence $P\in U\big(K_0(R)\big)$ by Lemma~\ref{temel1}(2) and a simple computation shows that $P^{-1}AP=\left[%
\begin{array}{cc}
 a & -b\\
 0 & d \\
\end{array}%
\right]$. By Lemma~\ref{lem5} and Lemma~\ref{lem6}, $A$ is
quasipolar.  $(2)$ is obvious from $(1)$ and Theorem~\ref{power}.
\end{proof}

Note that $M_2(\Bbb Z_{(2)})$ is not quasipolar but $K_0(\Bbb
Z_{(2)})$ is quasipolar by Corollary~\ref{cor3}.

\begin{thm}\label{teo5}Let $R$ be a commutative local ring with
$s=\sum\limits_{i=0}^{n-1}s_ix^i\in R[x]/(x^n)$, where $n\geq 1$.
The following are equivalent.
\begin{enumerate}
\item $K_s\big(R[x]/(x^n)\big)$ is quasipolar.
\item $K_{s_0}(R)$ is quasipolar.
\end{enumerate}
In particular, for all $1\leq m \leq n$,
$K_{x^m}\big(R[x]/(x^n)\big)$ is quasipolar.
\end{thm}

\begin{proof} Similar to the proof of Theorem~\ref{power}.
\end{proof}

\begin{thm}\label{teo4}Let $R$ be a commutative local ring. If $s\in R$ is nilpotent, then $K_s(R)$ is quasipolar.
\end{thm}

\begin{proof} Let $A\in K_s(R)$. We may assume that $A\notin U\big(K_s(R)\big)$ and $A\notin
\big(K_s(R)\big)^{qnil}$. According to Lemma~\ref{lem8}, $A$ is
similar to $\left[%
\begin{array}{cc}
 u & 1\\
 v & w \\
\end{array}%
\right]$ or $\left[%
\begin{array}{cc}
 w & 1\\
 v & u \\
\end{array}%
\right]$ where $u,v\in U(R)$ and $w\in J(R)$. Without loss of
generality, $A\sim B=\left[%
\begin{array}{cc}
 u & 1\\
 v & w \\
\end{array}%
\right]$. Note that $t^2-tr(B)t+det_s(B)=t^2-(u+w)t+(uw-sv)=0$ if
and only if
$\big((u+w)t\big)^2-(u+w)\big((u+w)t\big)+uw-sv=(u+w)^2[t^2-t+(uw-sv)(u+w)^{-2}]=0$.
In proof of \cite[Theorem 22]{TZ}, it is proved that the equation
$t^2-ut-w=0$ is solvable in $R$ for all $u\in 1 + J(R)$ and $w\in
J(R)$. Then the equation $t^2-t+(uw-sv)(u+w)^{-2}=0$ is solvable
in $R$ and so the equation $t^2-tr(B)t+det_s(B)=0$ is solvable in
$R$. In view of Theorem~\ref{teo2}, $A$ is quasipolar and so
$K_s(R)$ is quasipolar.
\end{proof}

For a commutative local ring $R$, $K_s(R)$ is quasipolar for some
non-nilpotent elements $s$ in $J(R)$, by Corollary~\ref{cor3}.

\begin{ex}\label{ex1} \rm{ If  $R=\Bbb Z_{(p)}[x]/(x^2)$, then
$K_s(R)$ is quasipolar if and only if $s\in R^{nil}$.}
\end{ex}

\pagebreak

\end{document}